\documentclass[11pt,leqno]{article}
\usepackage{latexsym,amssymb}
\usepackage{hyperref}
\usepackage{amsmath,amsthm}
\usepackage{amsfonts}
\usepackage[american]{babel}
\usepackage{mathrsfs}
\usepackage{psfrag}
\usepackage{todonotes}

\allowdisplaybreaks

\newcommand{\R}{\mathbb{R}}
\newcommand{\rn}{\mathbb{R}^n}
\newcommand{\Rno}{\mathbb{R}^n\setminus\{0\}}
\newcommand{\N}{\mathbb{N}}
\newcommand{\dive}{\text{div}}
\newcommand{\Dxi}{\nabla_\xi}
\newcommand{\Du}{\nabla u}

\newcommand{\barx}{\overline x}

\newcommand{\bd}{\partial}
\newcommand{\tr}{\text{tr}}
\newcommand{\scal}[2]{\,#1\cdot#2\,}

\newcommand{\Ln}{\mathcal L^n}

\newcommand{\ep}{\varepsilon}

\numberwithin{equation}{section}
\newtheorem{theorem}{Theorem}[section]

\newtheorem{lemma}[theorem]{Lemma}

\begin{document}

\title{
Wulff shape symmetry of solutions to overdetermined problems for Finsler Monge-Amp\`ere equations
} \frenchspacing

\author{
Andrea Cianchi\\
{\it Dipartimento  di Matematica ``U. Dini'', Universit\`a di
Firenze}\\ {\it Viale Morgagni 67/A, 50134 Firenze, Italy, e-mail:
andrea.cianchi@unifi.it}
\bigskip
\\
Paolo Salani\\
{\it Dipartimento  di Matematica ``U. Dini'', Universit\`a di
Firenze}\\ {\it Viale Morgagni 67/A, 50134 Firenze, Italy, e-mail:
paolo.salani@unifi.it} }

\date{}


\maketitle

\begin{abstract}
We deal with  Monge-Amp\`ere type equations modeled upon general anisotropic norms $H$ in $\rn$. An overdetermined problem for convex solutions to these equations is analyzed. The relevant solutions are subject to both a homogeneous Dirichlet condition and 
 a second boundary condition, designed on $H$, on the gradient image of the domain. The Wulff shape symmetry associated with $H$ of the solutions is established.
\end{abstract}


\section{Introduction and main result}\label{sec1}

The   standard Monge-Amp\`ere operator is formally defined, for a real-valued function on an open set $\Omega \subset \rn$, with $n \geq 2$, as
\begin{equation}\label{M}
Mu = \det (\nabla ^2 u).
\end{equation}
Here, $\nabla ^2 u$ denotes the Hessian $n\times n$--matrix of the second-order derivatives of $u$ and \lq\lq$\det$" stands for determinant. The operator $M$ lies at one endpoint of a family of so-called Hessian equations, whose opposite terminal is occupied by the Laplace operator. Besides the usual divergence form $\Delta u  ={ \rm div} (\nabla u)$, 
the latter can be expressed  as
\begin{equation}\label{delta}
\Delta u = \tr (\nabla ^2 u),
\end{equation}
where \lq\lq$\tr$" denotes trace.  

The structure of these classical operators
 is in a sense related to the use of the Euclidean norm in the ambient space $\rn$. For instance, the Laplacian emerges from the Euler equation of the Dirichlet energy integral,
defined in terms of the  Euclidean norm $|\nabla u|$ of the gradient $\nabla u$ of $u$.
 Replacing this norm with a  general anisotropic norm $H: \rn \to [0, \infty)$ results in the functional
\begin{equation}\label{dirichletH}
\int_\Omega E(\nabla u)\, dx,
\end{equation}
where $E: \rn \to [0, \infty)$ is the function given by
\begin{equation}\label{V}
E(\xi) = \frac{H(\xi)^2}2 \quad \text{for $\xi \in \rn$.}
\end{equation}
%
The metric in $\rn$ associated with a norm $H$ is called a Finsler metric in the literature.
The Finsler Laplacian $\Delta_H$ built upon $H$ is defined as the differential operator appearing in the  Euler equation for functional \eqref{dirichletH}.  In analogy with \eqref{delta}, it can be written in the form
\begin{equation}\label{deltaH}
\Delta_H u =      \tr \big(\nabla (\nabla_\xi E(\nabla u))\big).  
\end{equation}
Here and in what follows, the index $\xi$ for a differential operator denotes differentiation in the \lq\lq gradient variable". Definition \eqref{deltaH} is justified by its alternate  divergence form $\Delta_H u  = {\rm div} \big(H(\nabla u) H_\xi (\nabla u)\big)$, which surfaces in deriving the Euler equation of the functional \eqref{dirichletH}. With this regard, observe that
\begin{equation}\label{energy1}
-\int_\Omega u\nabla_H u\, dx = 2\int_\Omega E(\nabla u)\, dx
\end{equation}
provided that $u$ vanishes on $\partial \Omega$.
Plainly, definition \eqref{deltaH} recovers \eqref{delta} when $H$ is the Euclidean norm, since in this case $\nabla_\xi E(\xi)=\xi$ for $\xi \in \rn$.
%
\\
The operator $\Delta _H$ and its $p$-generalization, obtained analogously after replacing the exponent $2$ by any $p \in (1, \infty)$  in the functional \eqref{dirichletH},  have been  investigated under various respects. A sample of contributions on this subject is furnished by \cite{AFLT, Amar-Bellettini, AnCiFa, BeFeKa, DRSV1, DRSV2, FaLi1, FaLi2, FeK, KN, WaXi2}.

The Finsler the Monge-Amp\`ere operator $M_H$ 
 is  defined as
\begin{equation}\label{MH}
M_H u = \det  \big(\nabla (\nabla_\xi E(\nabla u))\big).
\end{equation}
Besides being suggested  by the mere replacement of the trace by the determinant on the right-hand side of equation \eqref{deltaH}, definition \eqref{MH} originates from the Euler equation of the  functional
\begin{equation}\label{energy}
\int_\Omega E(\nabla u)^{\frac {n+1}2}k_H(x)\, dx,
\end{equation}
where $k_H(x)$ denotes the Finsler Gauss curvature of the level set of $u$ at the point $x$.  This assertion can be verified via the identity
\begin{equation}\label{energyid}
- \int_\Omega  u \, M_Hu\, dx =   \frac {2^{\frac {n+1}2}}n
\int_\Omega E(\nabla u)^{\frac {n+1}2}k_H(x)\, dx,
\end{equation}
which holds for sufficiently smooth convex functions $u$ vanishing on $\partial \Omega$ and sufficiently smooth norms $H$.  When $H$ is the Euclidean norm, the integral on the left-hand side  is the customary energy functional of the Monge-Amp\`ere operator  and equation \eqref{energyid}  is classical.
Functional \eqref{energy} thus provides us with a natural generalization in the Finsler setting
and 
stands to the operator $M_H$ as  functional \eqref{dirichletH} stands to $\Delta_H$. 
We refer to \cite{DePGa, DePGaXi} for the derivation of equation \eqref{energyid} and properties of the operator $M_H$ in connection with ad hoc symmetrizations.

Formally, $ \det  \big(\nabla (\nabla_\xi E(\nabla u))\big)=
\det \big(\nabla_\xi^2 E(\nabla u) \nabla ^2 u\big)= \det \big(\nabla_\xi^2 E(\nabla u)\big)\det (\nabla ^2 u\big)$. Thereby,   $M_H$ can be regarded as a classical Monge-Amp\`ere operator with a gradient-depending coefficient, whose special structure depends on the norm $H$. We prefer to introduce it in the form \eqref{MH}, which allows for a definition of generalized solution in the sense of Alexandrov without any a priori regularity assumption on $H$.

Our focus here is on the symmetry of the solution to an overdetermined boundary value problem for the operator $M_H$. The interest in symmetry properties of solutions to overdetermined boundary value problems for partial differential equations was ignited half a century ago by the seminal paper  \cite{Serrin} by Serrin. A special case of his result concerns the Poisson equation, coupled with 
both a homogeneous Dirichlet condition and a constant Neumann condition at the boundary. It 
asserts that, if $\Omega$ is bounded and sufficiently smooth, and $c$ is any positive constant, then the problem
\begin{equation}\label{serrinpb}
\left\{\begin{array}{ll} \Delta u =1\quad&\text{in }\Omega\\
u=0 &\text{on }\bd\Omega
\\
|\nabla u| =c &\text{on }\bd\Omega
\end{array}\right.
\end{equation}
%
%
%
  admits a solution if and only if $\Omega$ is (up to translations and dilations) the  Euclidean unit ball $B$ centered at $0$. Moreover,
\begin{equation}\label{urad}
u(x)=\frac{|x|^2-1}{2}\qquad \quad\text{for } x\in
B.
\end{equation}

Over the years, Serrin's result has inspired a wealth of investigations on related questions -- 
%
see e.g. the surveys \cite{Magn, NiTr} on developments along this line of research. 
In particular, overdetermined boundary value problems for a family of Hessian-type equations are the subject of \cite{BNST}. 
The results of that paper include convex solutions to the Monge-Amp\`ere equation, with a constant right-hand side, on a bounded convex set $\Omega$ and subject to the same Dirichlet and Neumann boundary conditions as in  \eqref{serrinpb}. Observe that, because of the convexity of $\Omega$ and $u$, the latter condition is equivalent to requiring that
\begin{equation}\label{2condB}
\nabla u (\Omega) = B(c),
\end{equation}
where $B(c)$ denotes the Euclidean ball, centered at $0$, with radius $c$. This is a special case of the \lq\lq second boundary condition" in the theory of the Monge-Amp\`ere operator, so called as opposed to the Dirichlet boundary condition. In its general formulation, it amounts to imposing that 
$\nabla u (\Omega) = \Omega'$
for some bounded convex set $\Omega'$. 
It is also known as the \lq\lq natural boundary condition",  inasmuch as it arises naturally in the solution to the Monge-Kantorovich mass transportation problem.

The conclusion of \cite{BNST} is that, if the problem
\begin{equation}\label{CNST}
\left\{\begin{array}{ll} M u=1\quad&\text{in }\Omega\\
u=0 &\text{on }\bd\Omega
\\
\Du (\Omega) =B(c) 
\end{array}\right.
\end{equation}
admits a convex solution in a bounded convex domain $\Omega$, then (up to translations and dilations) the latter agrees with the Euclidean unit ball $B$, and the solution $u$ obeys \eqref{urad}. An alternate proof is offered in \cite{BNST1}.   An extension of this result to overdetermined problems for a class of fully nonlinear elliptic equations more general than that considered in  \cite{BNST1} can be found in \cite{SiSi} and rests upon a  different approach.

 The analysis of overdetermined problems in the Finsler ambient was initiated in \cite{Ci-Sa}, where 
a version of Serrin's theorem was established for (sufficiently smooth) general norms $H$. 
Loosely speaking, it  tells us that a symmetry result still holds, provided that the role of Euclidean balls is replaced by balls in the norm $H$ in the \lq\lq gradient variable" and balls in the 
the dual norm $H_0$ in the \lq\lq space variable". 
Balls according to the dual norm $H_0$ are usually said to have the Wulff shape associated with $H$. This terminology comes after G.Wulff, who, at the beginning of the last century, employed anisotropic geometric functionals built upon general norms $H$ in his mathematical theory of crystals \cite{Wu}. The functionals in question replace the standard perimeter of sets in $\rn$. They are defined as the integral, over the boundary of a set, of the function $H$ evaluated at the unit normal vector. Of course, the boundary and the unit normal vector of a set have to be properly defined according to geometric measure theory. Wulff-shaped balls are known to solve the corresponding isoperimetric problem among sets of prescribed Lebesgue measure \cite{Taylor}. 

More precisely, the paper  \cite{Ci-Sa} concerns the  problem
\begin{equation}\label{CiSa}
\left\{\begin{array}{ll} \Delta_H u =1\quad&\text{in }\Omega\\
u=0 &\text{on }\bd\Omega
\\
H(\nabla u) =c &\text{on }\bd\Omega.
\end{array}\right.
\end{equation}
 Under suitable regularity assumptions on the bounded domain $\Omega$, it asserts that a solution to problem \eqref{CiSa} exists if and only if $\Omega = B_{H_0}$  (up to translations and dilations). Moreover,   the solution $u$ is obtained on replacing the norm $|x|$  with $H_0(x)$ in equation \eqref{urad}. Here, $B_{H_0}$ denotes the unit ball, centered at $0$,  in the metric of $H_0$; balls with radius $c>0$  will be denoted by  $B_{H_0}(c)$.  Analogous notations are adopted for balls in the metric of $H$. 
%

In the last decade, further contributions appeared on overdetermined boundary value problems for the Finsler Laplacian and its $p$-Lapalcian version and, more generally,  on symmetry properties of solutions to problems involving these operators. A partial list  includes \cite{BaEn, Bi-Ci, Bi-Ci-Sa, CiFiRo, CoFaVa1, CoFaVa2, DeMaMiNe, DiPPoVa}. 

In the present paper, we complement the picture outlined above and show the Wulff shape symmetry of the solution to the Finsler Monge-Amp\`ere equation, when simultaneously subject to the homogeneous Dirichlet condition and a Wulff shape symmetric second boundary condition. This is the content of the following theorem. In its statement, the notation $C^2_+$ denotes the class of twice continuously differentiable functions whose Hessian matrix is everywhere positive definite.

\begin{theorem}\label{thmMAIN}
Let $\Omega$ be a  convex bounded open set in $\rn$.
Let $H$ be a norm in $\rn$
such that $H^2 \in C^2_+(\rn\setminus\{0\})$.  
Assume that there exists an Alexandrov convex solution $u$  
to the problem
\begin{equation}\label{eq85ter}
\left\{\begin{array}{ll} M_H u=1\quad&\text{in }\Omega\\
u=0 &\text{on }\bd\Omega
\\
\Du (\Omega) =B_H(c) 
\end{array}\right.
\end{equation}
for some constant  $c> 0$.
Then 
\begin{equation}\label{eq100}
\Omega=B_{H_0}\,
\end{equation}
and
\begin{equation}\label{apr14}
u(x)=\frac{H_0(x)^2-1}{2}\qquad \quad\text{for } x\in
B_{H_0},
\end{equation}
up to translations and dilations.
\end{theorem}

As observed above, owing to the convexity of $\Omega$ and $u$, the last condition in problem \eqref{eq85ter} can be equivalently formulated as the boundary condition
$$H(\nabla u) =c  \qquad \text{on }\bd\Omega.$$

Our approach to Theorem \ref{thmMAIN} departs from the original technique employed in \cite{Serrin}, which is based on a variant of the method of moving planes. This method was introduced by Alexandrov in his proof of the symmetry of bodies whose boundary has constant mean curvature and, after \cite{Serrin}, it was adapted to the proof of a variety of symmetry results for PDEs. We instead resort to arguments ultimately rooted in an alternate proof of Serrin's result by Weinberger \cite{We} and later developed in \cite{BNST}. They rely upon integral identities and inequalities involving the solution $u$ to problem \eqref{eq85ter}. The overall idea is that overdetermination causes all the inequalities to hold as equalities. This piece of information forces  $\nabla u$ to agree with the gradient of the right-hand side of equation \eqref{apr14}, whence the conclusion follows.

Besides other ingredients, the derivation of the relevant integral relations makes use of duality arguments pertaining to the theory of convex functions and sets. Specific properties linking the norm $H$ with its dual $H_0$, and the solution $u$ with its Young conjugate $\widetilde u$ enter the game. The regularity of the solution $u$ is critical in substantiating several steps of the argument. The nowadays classical $C^{1,\alpha}$ regularity theory by Caffarelli, as well as the $W^{2,1}$ regularity theory more recently inaugurated by De Philippis and Figalli, play a key role in this connection.

The paper is organized as follows. We begin by gathering definitions and some properties of the functions $H$  and $H_0$. The notion of Legendre conjugate and its basic properties are also recalled. This is the content of Section \ref{prel}. Section \ref{sol} is devoted to a precise formulation of problem \ref{eq85ter}. The definition of its solution in the framework of the theory of Monge-Amp\`ere type equations is discussed in that section, where its regularity properties of use for our analysis are presented as well. The proof of Theorem \ref{thmMAIN} is then accomplished in Section \ref{proof}.

\section{Convex functions and norms}\label{prel}

\subsection{Legendre conjugate}
Let $\Omega$ be a convex set and let $u : \Omega \to \R$ be a convex function. Its Legendre conjugate $\widetilde u : \rn \to \R$ is defined as
\begin{equation}\label{legendre}
\widetilde u (\xi) = \sup \{x\cdot \xi-u(x)\,:\, x\in\Omega\} \qquad \text{for $\xi \in \rn$.}
\end{equation}
If the function $u\in C^1(\Omega)$ and is strictly convex, then $\widetilde u \in C^1(\nabla u(\Omega))$ and is strictly convex  (see \cite[Theorem 26.5]{Rock}).
Moreover the map $\Du:\Omega\to\Du(\Omega)$ is invertible and
\begin{equation}\label{gradu}
\Dxi\widetilde u = (\Du)^{-1}\,.
\end{equation}
The very definition of the function $\widetilde u$ entails that 
\begin{equation}\label{young}
x \cdot \xi \leq u(x) + \widetilde u (\xi) \quad \text{for $x\in \Omega$ and $\xi \in \rn$.}
\end{equation}
Moreover, if $u\in C^1(\Omega)$ and  is strictly convex, then
\begin{equation}\label{young=}
x \cdot \xi = u(x) + \widetilde u (\xi) \quad \text{if  either $x= \Dxi \widetilde u (\xi)$ or $\xi = \nabla u(x)$.}
\end{equation}

\subsection{Norms in $\rn$.}
A function $H: \rn \to [0, \infty)$ is a norm in $\rn$ if and only if:
\begin{equation}\label{eq300}
H\text{ is convex},
\end{equation}
\begin{equation}\label{eq301}
H(\xi)\geq 0\,\,\text{ for} \,\,\xi \in \rn, \text{ and
}\,\,H(\xi)= 0\,\,\text{ if and only if }\,\,\xi = 0\,,
\end{equation}
\begin{equation}\label{eq302}
H(t\xi)=|t|H(\xi)\,\text{ for }\xi\in\rn\text{ and }t\in\R\,.
\end{equation}
The dual norm $H_0$ of $H$ is given by
\begin{equation}\label{H0}
H_0(x)=\max_\xi\frac{\scal{x}{\xi}}{H(\xi)} \qquad \text{for $x \in \rn$,}
\end{equation}
where the dot \lq\lq$\,\cdot\,$" stands for scalar product in $\rn$.
Conversely, $H$ is the dual norm of $H_0$, since
\begin{equation}\label{eq304}
H(\xi)=\max_x\frac{\scal{x}{\xi}}{H_0(x)} \qquad \text{for $\xi \in \rn$.}
\end{equation}
As mentioned above, given $r>0$, we denote by $B_H(r)$ the open ball, centered at $0$ and with radius $r$, in the metric generated by the norm $H$. Namely,
$$B_H(r)= \{\xi \in \rn: H(\xi) < r\}.$$
The ball $B_{H_0}(r)$ is defined analogously. If $r=1$, we shall simply write $B_H$ and $B_{H_0}$.
\\
One has that
\begin{equation}\label{eq305}
H_0\in C^1(\Rno)\,\text{ if and only if  $B_H$ is strictly convex,}
\end{equation}
see \cite[Corollary 1.7.3]{Schneider}.
Of course, an analogous property holds on interchanging the roles of  $H$ and $H_0$.
\\
If $H\in C^1(\Rno)$, then
\begin{equation}\label{eq306}
\Dxi H(t\xi)=\text{sign}(t)\Dxi H(\xi)\quad\text{for }\xi\neq0\text{ and }t\neq0\,.
\end{equation}
This is a consequence of property (\ref{eq302}),
Moreover,
\begin{equation}\label{eq307}
\scal{\xi}{\Dxi H(\xi)}=H(\xi)\quad\text{for }\xi\in\rn\,,
\end{equation}
where the left-hand side is taken to be $0$ when $\xi=0$.
\\
 Assume that
   $H\in C^1(\Rno)$ and $B_H$
is strictly convex.  Then, by \cite[Lemma 3.1]{Ci-Sa},
\begin{eqnarray}
H_0(\Dxi H(\xi))=1&\quad\text{for }\xi\in\Rno\,\label{eq2},
\end{eqnarray}
and
\begin{eqnarray}
H(\nabla H_0(x))=1&\quad\text{for }x\in\Rno\,\label{eq4}.
\end{eqnarray}
Moreover, under the same assumptions, the map $H\Dxi H:\rn\to\rn$ is invertible and
\begin{equation}\label{eq5}
H\Dxi H=\left( H_0\nabla H_0\right)^{-1}\,.
\end{equation}
Here, and in what follows, $H\Dxi H$ and $H_0\nabla H_0$ are
continued by $0$ at $0$.
 \\
If $H\in C^2(\Rno)$, then
\begin{equation}\label{eq308}
\Dxi^2H(t\xi)=\frac{1}{|t|}\Dxi^2H(\xi)\quad\text{for }\xi\neq 0\text{ and }t\neq 0\,,
\end{equation}
and
\begin{equation}\label{eq309}
\Dxi^2H^2(t\xi)=\Dxi^2H^2(\xi)\quad\text{for }\xi\neq 0\text{ and }t\neq 0.
\end{equation}
Observe that the matrix-valued function $\Dxi^2H^2$ is discontinuous at $0$,
unless it is constant. Yet, it is bounded, and hence
\begin{equation}\label{may1}
\text{$H\Dxi H$ is Lipschitz continuous in $\rn$,}
\end{equation}
inasmuch as   $\Dxi
H^2=2H\Dxi H$. Of course, a parallel property holds 
for $H_0$, provided that $H_0\in C^2(\Rno)$, namely
\begin{equation}\label{may2}
\text{$H_0\nabla H_0$ is Lipschitz continuous in $\rn$. }
\end{equation}
The properties enjoyed by the functions $H$ and $H_0$ mentioned so far are reflected in properties of the function $E$ defined by equation \eqref{V} and of the function
 $E_0$ defined by replacing $H$ with $H_0$ in the same equation. In particular, note that, by equations \eqref{eq302} and \eqref{eq2}, 
\begin{equation}\label{aug331}
E_0(\nabla _\xi E(\xi)) = E(\xi)  \quad \text{for $\xi \in \rn$.}
\end{equation}
Thanks to equation \eqref{eq307},
\begin{equation}\label{aug330}
\xi \cdot \nabla_\xi E(\xi) = 2 E(\xi) \quad \text{for $\xi \in \rn$.}
\end{equation}
Furthermore, if  $H^2\in C^2_+(\Rno)$, i.e. $E\in C^2_+(\Rno)$, then  
\begin{equation}\label{may5}
\Dxi^2E(t\xi)=\Dxi^2E(\xi)\quad\text{for }\xi\neq 0\text{ and }t\neq 0.
\end{equation}
Hence, there exist constants $\Lambda >\lambda >0$ such that
\begin{equation}\label{may6}
\lambda \leq |\Dxi^2E(\xi)| \leq \Lambda \quad\text{for }\xi\neq 0,
\end{equation}
and
\begin{equation}\label{sep104}
\lambda \leq \det (\Dxi^2E(\xi)) \leq \Lambda \quad\text{for }\xi\neq 0.
\end{equation}
Moreover, since 
$$\Dxi E = H \Dxi H \quad \text{and}  \quad \nabla E_0 = H_0 \nabla H_0,$$
the maps
\begin{equation}\label{apr50}
\nabla _\xi E : \rn \to \rn \quad \text{and} \quad \nabla E_0  : \rn \to \rn \quad \text{are Lipschitz continuous,}
\end{equation}
 and 
\begin{equation}\label{may9}
(\nabla _\xi E)^{-1} =   \nabla E_0.
\end{equation}
Especially,
\begin{equation}\label{apr50bis}
\nabla _\xi E : B_H \to B_{H_0} \quad \text{and} \quad \nabla E_0  : B_{H_0}  \to B_H.
\end{equation}
One can also verify that $E$ and $E_0$ are mutual Legendre conjugates. 
\\ A property of  Legendre conjugation ensures that, if    $E\in C^2_+(\Rno)$, then 
\begin{equation}\label{V0c2}
E_0\in C^2_+(\Rno).
\end{equation}

\section{Properties of solutions to Finsler Monge-Amp\`ere equations}\label{sol}

Let $H$ be any norm in $\rn$. 
A convex function 
$u: \Omega \to \R$ is a generalized  solution in the sense of Alexandrov to the equation
\begin{equation}\label{gener}
 M_H=1\qquad \text{in }\Omega
\end{equation}
if
\begin{equation}\label{gener1}
\mathcal L^n\big(\nabla _\xi E(\nabla u(\omega))\big) =\mathcal L^n(\omega)
\end{equation}
for every Borel set $\omega\subset \Omega$. Here,  $ \nabla u$  and $\nabla _\xi E$ are regarded as multi-valued maps,  which are well defined as subgradients, since both $u$ and $E$ are convex functions.
\\  Assume now that $H$ satisfies the additional assumptions of Theorem \ref{thmMAIN}.  Then $u$ is a  solution  to equation \eqref{gener} in the sense of Alexandrov if and only if it is an Alexandrov  solution to the equation
\begin{equation}\label{transporteq1}
 \Phi (\nabla u) \det(\nabla^2u)=1\quad\text{in }\Omega,
\end{equation}
where $\Phi : \rn \to [0, \infty)$ is the function defined as
\begin{equation}\label{Phi}
\Phi(\xi) = \begin{cases} \det(\nabla_{\xi}^2E(\xi)) & \quad  \text{ if $\xi \neq 0$}
\\  \inf_{\{\eta \neq 0\}}  \det(\nabla_{\xi}^2E(\eta))  & \quad   \text{if $\xi =0$.}
\end{cases}
\end{equation}
Recall that a convex function $u$ is an Alexandrov solution to equation \eqref{transporteq1} if
\begin{align}\label{transporteq}
\int_{\nabla u(\omega)} \Phi (\xi)\, d\xi =  \Ln(\omega)
\end{align}
for every Borel set $\omega\subset \Omega$. In order to verify this equivalence, notice that, by property  \eqref{apr50},  the map $\nabla _\xi E : \rn  \to \rn$ is bi-Lipschitz continuous. A change of variables via this map  yields
\begin{align}\label{apr20}
\int_{\nabla u(\omega)} \Phi (\xi)\, d\xi = \int_{\nabla u(\omega)} \det(\Dxi ^2 E(\xi))\, d\xi = \int_{\nabla _\xi E(\nabla u(\omega))}\, dx = \Ln \big(\nabla _\xi E(\nabla u(\omega))\big)= \Ln(\omega)
\end{align}
for every Borel set $\omega\subset \Omega$. Hence, the claimed equivalence follows.
\\ Next, let $\Omega$ be a bounded convex domain and consider the problem obtained by coupling  equation \eqref{gener}
with the second boundary condition $\Du(\Omega)=B_{H}$. In view of the above remarks, it can be formulated as
\begin{equation}\label{transport}
\left\{\begin{array}{ll} \Phi (\nabla u) \det(\nabla^2u)=1\quad&\text{in }\Omega
\\
\Du(\Omega)=B_{H}.
\end{array}\right.
\end{equation}
Besides that introduced by Alexandrov, another (even weaker) definition of a solution to the equation in \eqref{transport} is available. It was introduced by Brenier in his fundamental work on the Monge-Kantorovich mass transportation problem \cite{Brenier}. 
A convex function $u$ is a solution to the equation \eqref{transport} in the sense  of Brenier  if 
\begin{equation}\label{brenier}
\int _{\Du(\Omega)} h(\xi) \Phi(\xi)\, d\xi = \int _\Omega h(\nabla u(x))\, dx
%
\end{equation}
for every continuous function $h : \Du(\Omega) \to \mathbb R$, where $\nabla u$ is now regarded as a function defined a.e. in $\Omega$. Since,
by equation \eqref{sep104}, there exist constants $0<\lambda \leq \Lambda$ such that 
\begin{equation}\label{sep106}
\lambda \leq \Phi(\xi) \leq \Lambda \quad \text{for $\xi \in\rn$,}
\end{equation}
  the result of  \cite{Brenier} guarantees that   problem \eqref{transport}     admits a unique solution in this sense,  provided that the compatibility condition (dictated by the choice $h=1$ in \eqref{brenier})
$$\mathcal L^n(\Omega) = \mathcal L^n(B_{H_0})$$
is fulfilled.
\\
A result from \cite{Caff1992} provides us with the following information:
\\ (i) The Brenier solution to problem \eqref{transport} 
 is also a  solution  in the sense of Alexandrov; hence, we may refer to the (unique) solution $u$ to problem \ref{transport} without further specification in what follows;
\\ (ii) The Legendre conjugate $\widetilde u$ of $u$  is the Brenier and the Alexandrov solution to the problem
\begin{equation}\label{transportbis}
\left\{\begin{array}{ll} \det(\nabla^2\widetilde u) = \Phi (\xi) \quad&\text{in } B_H
\\
\nabla \widetilde u(B_H)=\Omega.
\end{array}\right.
\end{equation}
Hence, 
\begin{equation}\label{brenierbis}
 \int _\Omega f(x)\, dx = \int _{B_{H}} f(\nabla \widetilde u (\xi)) \Phi(\xi)\, d\xi 
\end{equation}
for every continuous function $f: \Omega \to \mathbb R$, where $\nabla \widetilde u$ is regarded as a function defined a.e. in $\Omega$. Moreover,  
\begin{equation}\label{alexbis}
\int _{\varpi}  \Phi(\xi)\, d\xi = \Ln \big(\nabla \widetilde u (\varpi)\big)
\end{equation}
for every Borel set $\varpi\subset B_H$, where $\nabla \widetilde u$ is considered a multi-valued map;
\\  (iii) $u$ and $\widetilde u$ are strictly convex;
\\ (iv) $u\in C^{1,\alpha}(\overline \Omega)$ and $\widetilde u \in  C^{1,\alpha}(\overline {B_{H}})$ for some $\alpha >0$; consequently, $\nabla u$ and    $\nabla \widetilde u$  are, in fact, single-valued maps, and 
 $\nabla u : \Omega \to B_{H}$ and $\nabla \widetilde u  : B_{H}\to \Omega$ are inverses of each other.
\\ 
In particular, 
the function $u$ fulfills the equation in \eqref{transport} a.e. in $\Omega$, and $\widetilde u$ fulfills the equation in \eqref{transportbis} a.e. in $B_H$.
\\
 Let us also notice that, thanks to equation \eqref{may9},
\begin{equation}\label{apr21}
 \Phi(\nabla E_0) {\rm det}(\nabla^2 E_0)  =  \det(\nabla_{\xi}^2E(\nabla E_0)) \,{\rm det}(\nabla^2 E_0)  = 1 \quad \text{a.e in $B_{H_0}$,} 
\end{equation}
and, as a consequence of equation \eqref{brenierbis} and the change of variable formula for Lipschitz maps,
\begin{align}\label{change}
\int_{\Omega} f(x)\, dx&  = \int_{B_H} f(\nabla \widetilde u (\xi)) \Phi(\xi)\, d\xi = \int_{B_{H_0}}  f(\nabla \widetilde u (\nabla E_0(y))    \det(\nabla_{\xi}^2E(\nabla E_0(y)))\, {\rm det}(\nabla^2 E_0(y))\, dy\\ \nonumber & = \int_{B_{H_0}}  f(\nabla \widetilde u (\nabla E_0(y))) \, dy
\end{align}
for every continuous function $f: \Omega \to \mathbb R$. 
\\ Since the functions 
$$ \Dxi E (\Du) : \Omega \to B_{H_0} \quad \text{and} \quad \nabla \widetilde u (\nabla E_0) : B_{H_0} \to \Omega$$
are inverses of each other,  an application of equation \eqref{change} with $f=g(\nabla_\xi E(\nabla u))$ yields
\begin{align}\label{change1}
\int_{B_{H_0}}  g(y)\, dy = \int_{\Omega} g(\nabla_\xi E(\nabla u(x)))\, dx
\end{align}
for every continuous function $g : B_{H_0} \to \mathbb R$.
\\ Since the function $\widetilde u$ solves problem \eqref{transportbis} and the function $\Phi$ fulfills inequalities \eqref{sep106}, \cite[Theorem 1.1]{Savin-Yu} -- a  global version for second boundary value problems for  Monge-Amp\`ere  equations of results of \cite{DePh-Fi}, \cite{DFS} and \cite{Schm} --   ensures that 
\begin{equation}\label{second}
  \widetilde u \in W^{2,1}(B_H).
\end{equation}
Moreover, we claim that  
\begin{equation}\label{secondu}
 u \in W^{2,1}(\Omega).
\end{equation}
This is again a consequence of \cite[Theorem 1.1]{Savin-Yu}, once one has noticed that $u$ is also an Alexandrov  solution to the problem
\begin{equation}\label{transportnew}
\left\{\begin{array}{ll}  \det(\nabla^2u)= \psi (x)\quad&\text{in }\Omega
\\
\Du(\Omega)=B_{H},
\end{array}\right.
\end{equation}
where the function $\psi :\Omega \to [0, \infty)$, defined as 
$$\psi (x) = \frac{1}{\Phi(\nabla u(x))} \quad \text{for $x \in \Omega$,}$$
satisfies the inequalities
$\frac 1\Lambda \leq \psi (x) \leq \frac 1\lambda$  for $x\in \Omega$.
The fact $u$ is   an Alexandrov  solution to problem \eqref{transportnew} 
 amounts to the equation
\begin{equation}\label{sep110}
\mathcal L^n(\nabla u(\omega)) = \int_\omega \psi (x)\, dx
\end{equation}
being fulfilled 
for every Borel set $\omega \subset \Omega$. Equation \eqref{sep110} formally follows from  \eqref{brenier} with the choice
$$h(\xi)=  \frac{\chi_{\nabla u(\omega)}(\xi)}{\Phi (\xi)} \quad \text{for $\xi \in B_H$}.$$
 A rigorous proof can be accomplished through a standard approximation argument for $h$ via convolutions with smooth mollifiers. Note that passage to the limit is justified by the dominated convergence theorem. A role is also played by the fact that, as a consequence of equation \eqref{transporteq}, $\mathcal L^n((\nabla u)^{-1}(K))=0$ for any set $K\subset  B_H$ such that $\mathcal L^n(K)=0$.
\\ Let us mention that, in fact, \cite[Theorem 1.1]{Savin-Yu} ensures that the space $W^{2,1}$ can even be  replaced by $W^{2,1+\varepsilon}$ for some  $\varepsilon>0$  in equations \eqref{second} and \eqref{secondu}. This additional piece of information will however not be needed in our proofs.

\section{Proof of Theorem  \ref{thmMAIN}}\label{proof}

In this section, we accomplish the proof of Theorem  \ref{thmMAIN}. A couple of steps rely upon 
a generalized version of Newton's inequality for matrices.  Since we have not been able to locate the relevant inequality in the literature, we state it in a lemma at the end of the section and provide a proof for completeness.
 
\begin{proof}[Proof of Theorem \ref{thmMAIN}]
Up to rescaling in $x$ and $u$, we may assume that $c=1$ in problem \eqref{eq85ter}. Thus, in the light of the discussion in Section   \ref{sol}, we may assume that $u$ is the Alexandrov solution to the problem
\begin{equation}\label{pb}
\begin{cases}
 \Phi (\nabla u) \det(\nabla^2u)=1&\quad \text{in }\Omega
\\
u=0 & \quad \text{on }\bd\Omega
\\
E(\Du)=\tfrac 12 &\text{ on }\bd\Omega,
\end{cases}
\end{equation}
where $\Phi$ is the function defined by \eqref{Phi}.
\\ The proof is split into several steps. 

\smallskip
\par\noindent
\emph{Step 1.} We have that
\begin{equation}\label{E1}
\Ln (\Omega)= \Ln (B_{H_0})\,.
\end{equation}
This is a consequence of equation \eqref{change} with $f=1$.

\smallskip
\par\noindent
\emph{Step2}.
Owing to property \eqref{secondu}, we have that $u \in W^{2,1}(\Omega)$.  Moreover, by our assumptions on $E$ and property \eqref{apr50}, the function $\Dxi E \in C^1(\rn \setminus \{0\}) \cap {\rm Lip} (\rn)$. Hence, the results of \cite{Marcus-Mizel} on the composition of vector-valued Sobolev functions ensure that
 \begin{equation}\label{aug300}
\Dxi E(\nabla u) \in W^{1,1}(\Omega)
\end{equation}
 and 
 \begin{equation}\label{aug301}\nabla (\Dxi E(\nabla u)) = \Dxi^2E(\nabla u)\nabla^2u  \quad \text{ a.e. in $\Omega$.}
\end{equation}
  Also, inasmuch as the solution $u$ is strictly convex,  $\nabla u$ vanishes only at the unique minimum point of $u$, whence 
\begin{equation}\label{sep100}
\det(\Dxi^2E(\nabla u)\nabla^2u)= \Phi(\nabla u)\det (\nabla^2u) \quad \text{ a.e. in $\Omega$.}
\end{equation}
Incidentally, let us notice that equation \eqref{sep100} would hold even if $u$ were not convex, inasmuch as  $\nabla^2u$ vanishes a.e. on a set where $\nabla u$ is constant, provided that $u$ is just in  $W^{2,1}(\Omega)$.
\\ Since the matrix $\nabla^2u$ is positive semidefinite and the matrix $\Dxi^2E$ is positive definite,
thanks to the generalized version of Newton's inequality of Lemma \ref{newton} below, to equation \eqref{sep100},  
 and to  the equation in \eqref{pb},
\begin{equation}\label{E2}
\Delta_H u=\text{div}(\Dxi E(\nabla u))=\text{tr}(\Dxi^2E(\nabla u)\nabla^2u)\geq n\det(\Dxi^2E(\nabla u)\nabla^2u)^{1/n}=n \quad \text{a.e. in $\Omega$.}
\end{equation}
 Moreover,  equality holds in the inequality  in \eqref{E2} if and only if
\begin{equation}\label{eqcondition}
\Dxi^2E(\nabla u)\nabla^2u= I,
\end{equation}
where $I$ denotes the identity matrix in $\mathbb R^{n\times n}$.

\smallskip
\par\noindent
\emph{Step 3. } We have that
\begin{equation}\label{E3}
2\int_{B_{H_0}}E_0(y)dy \geq -n\int_\Omega u\,dx\,,
\end{equation}
and equality holds in the inequality if and only if equality \eqref{eqcondition} holds a.e. in $\Omega$.
\\
To prove this assertion, we   make use of equation \eqref{change1} to   obtain that
%
$$
\int_{B_{H_0}}E_0(y)dy=\int_\Omega E_0(\Dxi E(\Du))\,dx.
$$
Hence, thanks to equation \eqref{aug331},
%
\begin{equation}\label{E3.1}
\int_{B_{H_0}}E_0(y)dy=\int_\Omega E_0(\Dxi E(\Du))\,dx=\int_\Omega E(\Du)\,dx\,.
\end{equation}
The definition of $\Delta_H$, the divergence theorem, equation \eqref{aug330} and the  Dirichlet boundary condition in \eqref{pb} yield
 \begin{align}\label{apr30}
-\int_{\Omega}u\,\Delta_Hu\,dx&=-\int_{\Omega}\dive(u\nabla _\xi E(\nabla u))- \Du\cdot\Dxi E(\Du)\,dx
\\ \nonumber
&=2\int_{\Omega} E(\Du)dx- \int_{\partial \Omega}u\nabla _\xi E(\nabla u)\cdot\nu\,d\mathcal H^{n-1} =2\int_{\Omega} E(\Du)dx.
\end{align}
Hence, by equation \eqref{E2}, 
\begin{equation}\label{apr33}
-n \int_{\Omega} u \,dx \leq 2 \int_{\Omega} E(\Du)\,dx.
\end{equation}
Inequality \eqref{E3} follows from \eqref{apr33}, via equation \eqref{E3.1}. The assertion about the case of equality in \eqref{E3} is a consequence of the case of equality in    \eqref{E2}.

\smallskip
\par\noindent
\emph{Step 4. } Here, we show that
\begin{equation}\label{E4}
\int_\Omega u\,dx=-\int_{B_{H_0}}\widetilde u(\nabla E_0)\,dy+\int_{B_{H_0}}\Dxi\widetilde u(\nabla E_0)\cdot\nabla E_0\,dy,
\end{equation}
where $\widetilde u$ denotes the Legendre conjugate of $u$.
\\ Equation \eqref{E4} follows via equation \eqref{change}, which entails that
\begin{align}\label{apr5}
\int_{\Omega } u\,dx & = \int_{B_{H_0}} u(\Dxi \widetilde u (\nabla E_0))\, dy = 
 \int_{B_{H_0}} \Dxi\widetilde u (\nabla E_0)\cdot \nabla E_0 - \widetilde u (\nabla E_0)\, dy.
\end{align}

\smallskip
\par\noindent
\emph{Step 5.} We claim that
\begin{equation}\label{E5}
\int_{B_{H_0}}\widetilde u(\nabla E_0)\,dy = -(n+1)\int_\Omega u\,dx\,.
\end{equation}
Equation \eqref{E5} is a consequence of the following chain:
\begin{align}\label{apr6}
\int_{B_{H_0}}\widetilde u(\nabla E_0)\,dy & = \int_\Omega \widetilde u (\nabla u)\,dx = \int _\Omega \nabla u \cdot x - u \, dx = \int _\Omega {\rm div} (u x) -n u -u \, dx 
\\ \nonumber & =\int_{\partial \Omega} u x\cdot \nu \, d\mathcal H^{n-1}(x) - (n+1) \int_\Omega u\, dx = - (n+1) \int_\Omega u\, dx.
\end{align}
Note that the first equality holds 
thanks to equation \eqref{change} and to equation \eqref{gradu},
%
 the second one by equation \eqref{young=}, the fourth one by the divergence theorem, and the last one by the Dirichlet boundary condition in \eqref{pb}.

\smallskip
\par\noindent
\emph{
Step 6.} The following inequality holds:
\begin{equation}\label{E6}
\int_{B_{H_0}}\Dxi\widetilde u(\nabla E_0)\cdot\nabla E_0\,dy\geq\frac{n}{2}\left(\Ln (B_{H_0})-2\int_{B_{H_0}}E_0\, dy\right).
\end{equation}
By property \eqref{second}, one has that
$
\widetilde u \in W^{2,1} (B_H).
$
Also, properties \eqref{apr50} and \eqref{may9} ensure that
$
\nabla E_0 : B_ {H_0} \to \Omega$ is a by-Lipschitz map.
%
Hence, the change of variables formula for Sobolev functions 
tells us  that 
\begin{equation}\label{aug310}
 \Dxi\widetilde u (\nabla E_0) \in W^{1,1} (B_{H_0}),
\end{equation}
 and 
\begin{equation}\label{aug311} \nabla (\Dxi\widetilde u(\nabla E_0)) = \nabla ^2_\xi\widetilde u (\nabla E_0) \nabla ^2E_0 \quad \text{a.e. in $B_{H_0}$.}
\end{equation}
 Moreover, by equation \eqref{apr21},
\begin{equation}\label{sep101}
 {\rm det}(\nabla ^2_\xi\widetilde u (\nabla E_0) \nabla ^2E_0)= \displaystyle \frac{\det(\nabla^2\widetilde u(\nabla E_0))}{\Phi (\nabla E_0)} \quad \text{a.e. in $B_{H_0}$.}
\end{equation}
\\ 
An application of the divergence theorem yields 
\begin{align}\label{apr7}
\int_{B_{H_0}} \Dxi\widetilde u (\nabla E_0) \cdot \nabla E_0\, dy & = \int_{\partial B_{H_0}} E_0 \,\Dxi\widetilde u (\nabla E_0) \cdot \nu \, d\mathcal H^{n-1} - 
\int_{B_{H_0}}  E_0\, {\rm div}( \Dxi\widetilde u (\nabla E_0))\, dy
\\ \nonumber & 
 = \frac 12\int_{\partial B_{H_0}}  \,\Dxi\widetilde u (\nabla E_0) \cdot \nu \, d\mathcal H^{n-1} - 
\int_{B_{H_0}}  E_0\, {\rm div}( \Dxi\widetilde u (\nabla E_0))\, dy
\\ \nonumber & 
 =  \frac 12 \int_{B_{H_0}}   (1- 2 E_0)\, {\rm div}( \Dxi\widetilde u (\nabla E_0))\, dy.
\end{align}
Now, observe that  $1- 2E_0 \geq 0$ in $B_{H_0}$. Furthermore, 
\begin{align}\label{apr8}
 {\rm div}( \Dxi\widetilde u (\nabla E_0)) = {\rm tr}(\nabla ^2_\xi\widetilde u (\nabla E_0) \nabla ^2E_0) & \geq n \, {\rm det}(\nabla ^2_\xi\widetilde u (\nabla E_0) \nabla ^2E_0)
 =
 n \displaystyle \frac{\det(\nabla^2\widetilde u(\nabla E_0))}{\Phi (\nabla E_0)} = n  \quad \text{a.e. in $B_{H_0}$.}
%
%
\end{align}
Here, we have made use of equation \eqref{aug311}, of the version of Newton's inequality from Lemma \ref{newton},  of equation \eqref{sep101}, and of the equation in \eqref{transportbis}. Notice that the application of Lemma \ref{newton} is justified since the matrix $\nabla ^2_\xi\widetilde u$ is positive semidefinite and, by property \eqref{V0c2}, the matrix $\nabla ^2E_0$ is positive definite. 
\\
From equations \eqref{apr7} and \eqref{apr8} one deduces that 
\begin{align}\label{apr35}
\int_{B_{H_0}} \Dxi\widetilde u (\nabla E_0) \cdot \nabla E_0\, dy & \geq 
 \frac n2 \int_{B_{H_0}}   (1- 2E_0)\, dy.
\end{align}
This establishes inequality \eqref{E6}.

\smallskip
\par\noindent
\emph{
Step 7.} We have that
\begin{equation}\label{E8}
2\int_{B_{H_0}}E_0\,dy \leq -n\int_\Omega u\,dx\,.
\end{equation} 
To begin with, observe that 
merging \eqref{E5} and \eqref{E6} into \eqref{E4} implies that
\begin{equation}\label{E7}
-n\int_\Omega u\,dx\geq\frac{n}{2}\left(\Ln (B_{H_0})-2\int_{B_{H_0}}E_0\,dy\right)
\end{equation}
On the other hand
\begin{equation}\label{apr12}
\Ln (B_{H_0})=\frac{2(n+2)}{n}\int_{B_{H_0}}E_0\,dy\,,
\end{equation}
To verify this equality, notice that
\begin{align}\label{apr10}
\Ln (B_{H_0}) & = \int _{\{E_0(x)\leq \frac 12\}}dx 
= \int_0^{\frac 12} \int_{  \{E_0(x)=t\}}\frac {d\mathcal H^{n-1}(x)}{|\nabla E_0(x)|}\, dt 
\\ \nonumber & =
 \int_0^{\frac 12} \int_{  \{E_0(x/\sqrt t)=1\}}\frac {d\mathcal H^{n-1}(x)}{|\nabla E_0(x)|}\, dt 
=
 \int_0^{\frac 12} t^{\frac{n-1}2}\int_{  \{E_0(y)= 1\}}\frac {d\mathcal H^{n-1}(y)}{|\nabla E_0(y\sqrt t)|}\, dt 
\\ \nonumber & =
 \int_0^{\frac 12}t^{\frac{n-1}2-\frac 12}\int_{  \{E_0(y)= 1\}}\frac {d\mathcal H^{n-1}(y)}{|\nabla E_0(y)|}\, dt
= \frac 1{n2^{\frac n2-1}} \int_{ \{E_0=1\}}\frac {d\mathcal H^{n-1}}{|\nabla E_0|}\, dt,
\end{align}
where the second equality holds by the coarea formula, the third one since $E_0$ is homogeneous of degree $2$, the fourth one by the area formula, and the fifth one since the function $|\nabla E_0|$ is homogeneous of degree $1$. Analogously,
\begin{align}\label{apr11}
\int_{B_{H_0}}E_0\,dy 
= \int_0^{\frac 12} t\int_{ \{E_0=t\}}\frac {d\mathcal H^{n-1}}{|\nabla E_0|}\, dt =  \int_0^{\frac 12} t^{\frac n2}\int_{\{E_0=1\}}\frac {d\mathcal H^{n-1}}{|\nabla E_0|}\, dt =  \frac 1{(n+2)2^{\frac n2}} \int_{ \{E_0=1\}}\frac {d\mathcal H^{n-1}}{|\nabla E_0|}\, dt,
\end{align}
Combining equations \eqref{apr10} and \eqref{apr11} yields \eqref{apr12}. Inequality \eqref{E8} follows from equations \eqref{E7} and \eqref{apr12}.

\smallskip
\par\noindent
\emph{
Step 8.}
Conclusion. 
\\ Coupling  inequality \eqref{E8} with the reverse inequality \eqref{E3} implies that
\begin{equation}\label{apr13}
2\int_{B_{H_0}}E_0\,dy =-n\int_\Omega u\,dx\,.
\end{equation} 
This forces all the inequalities derived above to hold as equalities. Especially, equation
 \eqref{eqcondition} holds a.e. in $\Omega$. Namely,
$$\nabla (\nabla_\xi E(\nabla u))
=  I \quad \text{a.e. in $\Omega$.}$$
Hence, 
there exists $\barx\in\rn$ such that
\begin{equation}\label{apr15}
\nabla_\xi E(\nabla u) = x-\barx\quad\text{ for }x\in\Omega\,,
\end{equation}
and, by equation (\ref{may9}),
\begin{equation}\label{apr16}
\nabla u=\nabla E_0(x-\barx)
\quad \quad\text{ for }x\in\Omega.
%
%
%
\end{equation}
Equation \eqref{apr14} follows form (\ref{apr16}) and the Dirichlet boundary condition in \eqref{pb}.
\end{proof}

\begin{lemma}\label{newton}
Assume that the matrix  $A\in \R^{n\times n}$ is symmetric and positive definite and that the matrix  $B\in \R^{n\times n}$ is symmetric and positive semidefinite. Then
\begin{equation}\label{newton1}
(\det AB)^{\frac 1n} \leq  \frac{{\rm tr} (AB)}n.
\end{equation}
Moreover,   equality holds in \eqref{newton1} if and only if $AB= \lambda I$ for some constant $\lambda \geq 0$.
 \end{lemma}
\begin{proof}
Our assumptions on $B$ ensure that there exist an orthogonal matrix $U$ and a diagonal matrix $\Lambda = {\rm diag} \{\lambda _1, \dots, \lambda _n\}$, where $\lambda _1, \dots, \lambda _n\geq 0$ are the eigenvalues of $B$, such that
$\Lambda = U^T B U$. Hence $B=U\Lambda U^T$.
\\ Set $Q= U^TAU$ and $Q=(q_{ij})$. Then
\begin{align}\label{newton2}
\frac{\tr (AB)}n&  = \frac{\tr (AU\Lambda U^T)}n=  \frac{\tr (Q\Lambda)}n = \frac{\sum_{i=1}^n\lambda_i q_{ii}}n \geq \big(\Pi_{i=1}^n \lambda_i q_{ii}\big)^{\frac 1n} =  \big(\Pi_{i=1}^n q_{ii}\big)^{\frac 1n}\big(\Pi_{i=1}^n \lambda_i\big)^{\frac 1n} 
\\ \nonumber & =  \big(\Pi_{i=1}^n q_{ii}\big)^{\frac 1n}  (\det B)^{\frac 1n} \geq (\det Q)^{\frac 1n} (\det B)^{\frac 1n} = (\det A)^{\frac 1n} (\det B)^{\frac 1n} =   (\det AB)^{\frac 1n}.
\end{align}
Note that the first inequality in \eqref{newton2} holds by the arithmetic-geometric mean inequality, with equality if and only if 
\begin{equation}\label{newton3}
\lambda_1q_{11} = \dots = \lambda_{nn}q_{nn}.
\end{equation}
Moreover, the second inequality holds since $Q$ is a symmetric positive definite matrix \cite[Theorem 7.8.1]{Horn-Johnson}, with equality if and only if
\begin{equation}\label{newton4}
q_{ij} = 0 \quad \text{if $i\neq j$.}
\end{equation}
Inequality \eqref{newton1} agrees with  \eqref{newton2}. 
\\ In particular, if equality holds in \eqref{newton1}, then both equations \eqref{newton3} and \eqref{newton4} hold. The latter implies that the matrix $Q\Lambda$ is diagonal. Coupling this piece of information with the former, we deduce that $Q\Lambda$ is a multiple of $I$. Therefore
\begin{align}\label{newton5}
U^TABU= U^TAUU^TBU = Q \Lambda = \frac{\tr (Q\Lambda)}n I = \frac{\tr (AB)}n I.
\end{align}
Hence, 
$AB = \lambda I$, with $\lambda =\frac{\tr (AB)}n$.  Conversely, if $AB = \lambda I$ for some $\lambda \geq 0$, then equality trivially  holds in \eqref{newton1}.
\end{proof}
 
 \section*{Data availability}
 
Data sharing not applicable to this article as no datasets were generated or analysed during the current study.

  \section*{Compliance with Ethical Standards}\label{conflicts}

\smallskip
\par\noindent 
{\bf Funding}. This research was partly funded by:   
\\ (i) Research Project 201758MTR2  of the Italian Ministry of Education, University and
Research (MIUR) Prin 2017 ``Direct and inverse problems for partial differential equations: theoretical aspects and applications'';   
\\ (ii) GNAMPA   of the Italian INdAM - National Institute of High Mathematics (grant number not available).

\smallskip
\par\noindent
{\bf Conflict of Interest}. The authors declare that they have no conflict of interest.


\end{document}

Wang, Guofang; Xia, Chao A characterization of the Wulff shape by an overdetermined anisotropic PDE. Arch. Ration. Mech. Anal. 199 (2011), no. 1, 99–115.

When looking at their proof, we could 
not understand a final step at the end of section 4, where they assert 
that the point "x_0 must belong to the interior of C" (we had also 
tried an approach via a P-function, but we found some obstacles, 
including this one). In a reply message, they recognized that some 
approximation argument was needed (but I am not sure that this can be 
done). However, the final published version of their paper is the same 
as that of the preprint.

sovratedt

Silvestre, Luis; Sirakov, Boyan Overdetermined problems for fully nonlinear elliptic equations. Calc. Var. Partial Differential Equations 54 (2015), no. 1, 989–1007. 

 Bianchini, Chiara; Henrot, Antoine; Salani, Paolo An overdetermined problem with non-constant boundary condition. Interfaces Free Bound. 16 (2014), no. 2, 215–241.

 Brandolini, B.; Gavitone, N.; Nitsch, C.; Trombetti, C. Characterization of ellipsoids through an overdetermined boundary value problem of Monge-Ampère type. J. Math. Pures Appl. (9) 101 (2014), no. 6, 828–841. 

Dai, Qiuyi; Shi, Feilin Variational formula and overdetermined problem for the principal eigenvalue of the k-Hessian operator. J. Differential Equations 255 (2013), no. 11, 4136–4148. 

Fragalà, Ilaria Symmetry results for overdetermined problems on convex domains via Brunn-Minkowski inequalities. J. Math. Pures Appl. (9) 97 (2012), no. 1, 55–65.

sovradeterminati aniso

 Weng, Liangjun An overdetermined problem of anisotropic equations in convex cones. J. Differential Equations 268 (2020), no. 7, 3646–3664. 

Barbu, Luminita; Enache, Cristian On a free boundary problem for a class of anisotropic equations. Math. Methods Appl. Sci. 40 (2017), no. 6, 2005–2012.

Ciraolo, Giulio; Roncoroni, Alberto Serrin's type overdetermined problems in convex cones. Calc. Var. Partial Differential Equations 59 (2020), no. 1, Paper No. 28, 21 pp.

simmetria

 Farkas, Csaba; Kristály, Alexandru; Varga, Csaba Singular Poisson equations on Finsler-Hadamard manifolds. Calc. Var. Partial Differential Equations 54 (2015), no. 2, 1219–1241.

Wang, Guofang; Xia, Chao A Brunn-Minkowski inequality for a Finsler-Laplacian. Analysis (Munich) 31 (2011), no. 2, 103–115.

survey

\newpage

An archetypal result in the theory of elliptic PDE's asserts that,
for any bounded open set $\Omega$ in $\rn$, $n\geq 2$, there
exists a unique weak solution $u$ in the Sobolev space
$W^{1,2}_0(\Omega)$ to the Dirichlet problem
\begin{equation}\label{eq81}
\left\{\begin{array}{ll} -\Delta u=1\quad&\text{in }\Omega\\
\\
u=0\quad&\text{on }\bd\Omega\,
\end{array}
\right.
\end{equation}
for the Laplace operator $\Delta$. One proof of this statement
leans upon the fact
that (\ref{eq81}) is the Euler equation of the minimum problem
\begin{equation}\label{eq82}
\min_{u\in
W^{1,2}_0(\Omega)}\int_\Omega\left(\frac{1}{2}|\Du|^2-u\right)\,dx\,,
\end{equation}
involving a strictly convex and differentiable functional. Here,
$|\xi|$ denotes the Euclidean norm of $\xi\in\rn$, and $\nabla$
stands for the gradient operator with respect to the $x$ variable.
\par\noindent
In fact, problems (\ref{eq81}) and (\ref{eq82}) are completely
equivalent, and the latter admits a (unique) solution, since the
functional is trivially weakly lower semicontinuous and coercive
in $W^{1,2}_0(\Omega)$.
\par
It is well-known that the boundary condition on $u$ in problem
\eqref{eq81} can be generalized to  nonhomogeneous Dirichlet
conditions, and replaced by  analogous Neumann conditions, or even
by  mixed boundary conditions. However, Dirichlet and Neumann
conditions cannot be imposed simultaneously, unless they are
suitably related to the geometry of $\Omega$. In this connection,
a celebrated result by Serrin \cite{Serrin} entails that if the
solution $u$ to (\ref{eq81}) happens to fulfill the additional
boundary condition
\begin{equation}\label{eq83}
|\Du|=C\quad\text{ on }\bd\Omega
\end{equation}
for some $C>0$, and $\bd\Omega$ is sufficiently smooth, then
$\Omega$ is necessarily a ball and
$$u(x)=\frac{r^2-|x|^2}{2n}$$ for some $r>0$ (up to translations).

Starting with the paper by Weinberger \cite{We}, a number of
results on overdetermined problems soon  appeared in the wake of
\cite{Serrin},  dealing with alternate proofs, extensions to more
general (nonlinear) differential operators, and variants of the
boundary conditions. A partial list of papers on this and related
topics includes \cite{Alessandrini-Garofalo, BNST, BNST2,
Farina-Kawohl, Fragala-Kawohl-Gazzola, Garofalo-Lewis,
Payne-Philippin, Philippin, Schaefer}.
\par
The purpose of the present paper is to merge  Serrin's  theorem in
a general symmetry principle for solutions to overdetermined
elliptic problems, where the relevant symmetry is not necessarily
the spherical one.
The underlying idea of our contribution is that a symmetry result
holds for  any overdetermined problem involving an anisotropic
elliptic operator with quadratic growth, from a suitable class,
provided that the additional boundary condition imposed on the
gradient of the solution matches the structure of the differential
operator. The resulting symmetry of the domain and of the solution
reflects, in turn, the (anisotropic) symmetry of the operator.
Besides extending it, we hope that our conclusions will provide
further insight on Serrin's original result, on clarifying the
interplay between differential and geometric properties in
overdetermined problems.
\par
Let us mention that the study of problems with the kind of
anisotropy  considered in this paper has a quite long history in
the calculus of variations, and goes back at least to Wulff
\cite{Wu}, who employed anisotropic geometric functionals
in the mathematical theory of crystals. More recent developments
can be found in \cite{Amar-Bellettini, Dacorogna,
Esposito-Fusco-Trombetti, Fonseca, Fonseca-Muller, Taylor}. Other
related variational problems and  PDE's
are contained e.g. in \cite{Alvino-Ferone-Lions-Trombetti,
Aronson-Crandall, Belloni-Ferone-Kawohl, Cordero-Nazaret-Villani,
GGH, GigaR}.